\documentclass[11pt, oneside]{article} %

\usepackage[margin=.7in]{geometry}

\usepackage{graphicx}
\usepackage{multicol,multirow}
\usepackage{amsmath,amssymb,amsfonts}
\usepackage{mathrsfs}
\usepackage{amsthm}
\usepackage[T1]{fontenc}
\usepackage{stmaryrd}

\usepackage{comment}

\newtheorem{theorem}{Theorem}[section]

\newtheorem{proposition}[theorem]{Proposition}
\newtheorem{cor}[theorem]{Corollary}

\newtheorem{definition}{Definition}
\theoremstyle{definition}

\numberwithin{equation}{section}

\newcommand{\R}{\mathbb{R}}

\newcommand{\N}{\mathbb{N}}

\newcommand{\dif}{\mathrm{d}}
\DeclareMathOperator{\supp}{supp}

\title{Smoothing by, and eccentric smoothing of, compactly supported RBFs}

\author{T. Hangelbroek\thanks{Department of Mathematics,University of Hawai`i -- M\=anoa}, 
C. Rieger\thanks{Department of Mathematics and Computer Science, Philipps-Universit\"at Marburg}
}

\begin{document}
\maketitle

\abstract{We consider compactly supported RBFs having algebraically decaying Fourier transforms. Here we focus especially on generalized Wendland RBFs and their modification by making them smoother away from zero, a process we call
eccentric smoothing.
Specifically, we consider mapping properties
of  the integral operators (sometimes known as covariance operator) for these novel type RBFs. Moreover, we show that eccentric smoothing makes wavelet-inspired compression technique for the kernel matrix feasible.
\let\thefootnote\relax
\let\thefootnote\relax
}

\section{Introduction}\label{sec:Intro}
A  compactly supported RBF
is  radially symmetric function $\phi:\R^d\to \R$, with $\mathrm{supp}\,\phi\subset B(0,1)$
which is positive definite: meaning
that for any finite set $X=\{x_1,\dots x_n\}$, the collocation
matrix $\Phi_X =\bigl(\phi(x_j-x_k)\bigr)_{j,k\le n}$
is strictly positive definite.

In this article, we consider compactly supported RBFs  which have 
have finite global smoothness $\phi\in C^{\lambda_1}(\R^d)$.
Our main object of study is to consider  the
(higher) smoothness away from the origin: 
$\phi\in C^{\lambda_2}\bigl(\R^d \setminus \{0\}\bigr)$ for some $\lambda_2>\lambda_1$, possibly infinity.
When $\lambda_2$ is finite, we are especially interested in the case  that it is sharp (so that $\phi\notin C^{\lambda_2+\epsilon}\bigl(\R^d \setminus \{0\}\bigr)$ for $\epsilon>0$).
This functions as a useful secondary parameter for the RBF,
which we call the eccentric smoothing parameter.

The advantage of an increased eccentric smoothing parameter was noted in \cite{HR-Extending}, where it was used to obtaining high order convergence rates  in high order Sobolev norms for approximation and
interpolation with compactly supported RBFs.
There are other advantages.

\paragraph{Compression of the collocation matrix}
Recently, the compressibility of the matrix $\Phi_X$, which is full (despite the compact support of $\phi$)
and becomes ill-conditioned for large or poorly arranged point sets $X$, has been a favorable outcome of 
the study of {\em samplets}.

For certain RBFs, 
 after a change of basis (represented by the matrix $\mathcal{T}$), the collocation matrix
$\Phi_X^{\Sigma}= \mathcal{T}\Phi_X \mathcal{T}^{-1}$ can be rendered
sparse with little cost. The ability to compress the matrix relies on the smoothness of the kernel away from the origin.

Specifically, it relies on the {\em asymptotic smoothness} of the kernel.
This is a concept  which has migrated from the theory of compression of $\mathcal{H}$ 
matrices \cite[Section 4.3.4]{Hackbusch}.
The original condition is very strong condition and compactly supported functions cannot satisfy it.
However, a weaker version, namely asymptotic smoothness of finite order, has been 
considered for positive definite kernels, 
where it is key to producing wavelet-like bases
 (i.e., samplets)
 for kernel spaces in \cite{HMSS,HM}.  
 Indeed,  for globally supported RBFs, this type of compression has been used
in \cite{AKMW}
to produce an effective multiscale algorithm for RBF interpolation with 
computational cost $\mathcal{O}(N\log N)$, where $N=\Xi$, and $\Xi$ is assumed to be {\em quasi-uniform}.
 This notion of asymptotic smoothness provides sufficient conditions  
 for compressibility of the transformed collocation matrix -- we 
 expect that the smoothness away from the origin provided by Proposition \ref{fattening}
 will allow compactly supported kernels to be used effectively in this setting.

\paragraph{Regularity of the integral operator}
The  integral operator of $\phi$ 
plays an important role in kernel approximation, machine learning  and statistics
\cite{CS,LGZ, Berlinet},
the square root of the operator acting on $L_2(\sigma)$ is the kernel's  ``native space'', the eigenvectors
of the operator are fundamental to  Mercer  and Karhunen-Lo{\`e}ve expansions (see \cite[Ch. 3]{CS}
\cite[Ch. 2.3]{Berlinet} for background). The range
of the operator on $L_2(\sigma)$ is the doubling class considered in \cite{schaback_doubling}.

When $\phi$ is compactly supported,
convolution with tempered distributions 
is well defined and  continuous,
and the action
 $\phi*:\mathcal{S}'(\R^d)\to \mathcal{S}'(\R^d):f\mapsto \phi*f$ can be described by way of the Fourier transform (it is a multiplier).
However, if $\lambda_2$ is finite,
its {\em symbol} $a:(x,\xi) \mapsto \hat{\phi}(\xi)$, considered as a map on $\R^d\times \R^d$,
  does not belong to any of the H{\"o}rmander  symbol classes 
$$S_{\rho,\delta}^{\tau}(\R^d)=
\bigl\{
a\in C^{\infty}(\R^d\times \R^d)
\mid
 (\forall \alpha, \beta)\, (\exists C_{\alpha,\beta})\  |D_x^{\beta}D_{\xi}^{\alpha} a(x,\xi)| \le  C_{\alpha,\beta}
  \llbracket \xi\rrbracket^{\tau - \rho|\alpha|+\delta||\beta|}
\bigr\}.$$
with $0\le \delta<\rho\le 1$
(see \cite[Ch. 7]{Taylor2} for background).
Here we  use $\llbracket \xi\rrbracket := (1+|\xi|^2)^{1/2}$.
In short, the integral operator fails to be a H{\"o}rmander class  pseudodifferential operator.
This fact  can be verified in  the following two ways:
\begin{enumerate}
\item  if $j$ is sufficiently large, then
 $x\mapsto |x|^{2j} \phi(x)$ is  in $C^{\lambda_2}(\R^d)$, 
but not smoother. Thus  for any $j>(\lambda_2-\lambda_1)/2$, $|\Delta_{\xi}^{j}\widehat{\phi}(\xi)|$ cannot decay like
$|\xi|^{-\mu}$ for any $\mu> \lambda_1+d+1$.
\item if $\phi*$ were a pseudodifferential operator with symbol in $S_{\rho,\delta}^{\tau}(\R^d)$
with $\rho>0$ and $\delta<1$, then it would have the {\em pseudolocal property}: for any
compactly supported distribution $T$,
$\mathrm{sing\, supp} (\phi* T) \subset \mathrm{sing\, supp} (T)$ (see \cite[Ch 7.2]{Taylor2}). 
This fails because
$$\mathrm{sing\, supp}(\phi* \delta) = \mathrm{sing\, supp}( \phi) = \{0\} \cup \partial B(0,1) \supsetneq \{0\} =\mathrm{sing\, supp} (\delta).$$
\end{enumerate}
As a consequence, understanding the mapping properties of 
 $\phi*$  on 
 Sobolev spaces 
$W_p^s(\R^d)$ when $p\neq 2$
(or other classes like Besov and Triebel-Lizorkin spaces), or of the map
$f\mapsto \int_{\R^d} f(y) \phi(x-y) \dif\sigma(y)$
 on $L_p(\sigma)$ when $\sigma$ is a finite Borel measure
remains difficult. 
Both reduce to understanding how $\phi*$ acts on tempered distributions. This stands in sharp contrast to other well-known
(globally supported) RBFs having analytically similar Fourier transforms, such as Mat{\'e}rn kernels and surface splines, which induce pseudodifferential operators \cite{HRW}.

\paragraph{Outline}
In this brief article, we demonstrate how to modify, simply, 
standard constructions of compactly supported RBFs to produce 
compactly supported
RBFs with similar analytic properties
but which are smooth away from origin.
This is Proposition \ref{fattening}.
This yields kernels with {\em asymptotic smoothness} of finite order.
 The infinitely smooth kernels constructed using Proposition \ref{fattening} 
 have integral operators which 
 are properly supported
 pseudo-differential operators with symbols in the H{\"o}rmander class.
 This is demonstrated in Proposition \ref{symbol}. As a consequence, 
 these integral operators have prescribed smoothing properties.

\section{Compactly supported RBFs}
There are a number of constructions of compactly supported positive definite functions:
\cite{Buhmann, Johnson, Schaback, Wend, Wu}.
See also \cite{HubJag} for a recent  overview. 

In this article, we focus on a 
general framework which includes
 Wendland's original compactly supported RBFS
introduced in \cite{Wend} and further discussed in \cite[Chapter 9]{Wendland_book}
as well as
``generalized Wendland"   functions, 
introduced and considered in \cite{ChHub}.
 We exploit the analytic properties collected in
\cite[Section 5.2]{HR-Extending}, which can also be readily found in the literature (specifically in \cite{ChHub} 
and \cite[Chapter 10]{Wendland_book}).

A continuous, positive definite RBF  $\phi:\R^d\to \R$ is polyhomogeneous 
if there is a univariate polynomial $p$ so that
\begin{equation}
\label{Wend}
\phi(x) = \begin{cases}p(|x|), &|x|< 1\\ 0&|x|\ge 1.\end{cases}
\end{equation}
By splitting $p$ into even and odd parts and   expanding the odd part in monomials, we
have 
$$p(r)= \sum_{\ell=0}^M c_{\ell} r^{{\lambda_1} + 1+2\ell} + q(r^2),$$
for some univariate polynomial $q$, integers $0\le \lambda_1\le M$ and coefficients $c_{\ell}$ 
with $c_0\neq 0$. In this case, $\lambda_1$ is the {\em global smoothness} of $\phi$ in that
$\phi\in C^{{\lambda_1},1}(\R^d)$ (the partial derivatives of order ${\lambda_1}$  are Lipschitz);
but $\phi\notin C^{{\lambda_1}+1}(\R^d)$.

Define $\R^d_* := \R^d\setminus\{0\}$. A consequence of (\ref{Wend}) is that   there is 
a positive integer $\lambda_2$, the {\em eccentric smoothness}, for which
$\phi|_{\R^d_*}\in C^{{\lambda_2}}(\R^d_*)\setminus C^{{\lambda_2}+1}(\R^d_*)$.
By \cite[Lemma 9.10]{Wendland_book}, one has $\lambda_2>\lambda_1+\lfloor d/2\rfloor$.
This brings us to a general form for compactly supported RBFs we will use throughout this note:
\begin{definition} 
\label{gen_wend}
Let $\phi:\R^d\to \R$ be a compactly supported RBF satisfying (\ref{Wend}) with global smoothness $\lambda_1$,
and {\em eccentric smoothness}   $\lambda_2>\lambda_1 + \lfloor d/2\rfloor$.
Suppose furthermore that  there are constants $0<C_1\le C_2<\infty$ so that
\begin{equation}
\label{fourier}
C_1 \llbracket \xi \rrbracket^{-2t}\le \widehat{\phi}(\xi) \le C_2  \llbracket \xi \rrbracket^{-2t}
\end{equation}
with $ t := ({\lambda_1}+d+1)/2$.
\end{definition}
The original compactly supported RBFs of Wendland given
in \cite{Wend}
and further discussed in
\cite[Chapter 9]{Wendland_book}
satisfy the requirements of Definition \ref{gen_wend}
with $\lambda_2 = \lambda_1 +\lfloor\frac{d-1}{2}\rfloor$.
Constructions 
with larger values of $\lambda_2$
 can be found in \cite[Section 5.2]{HR-Extending} and  \cite{ChHub}.
In particular,
the class of {\em generalized Wendland functions} $\phi_{\mu, \alpha}$ 
defined in \cite[Eqn. 1.3]{ChHub}
satisfy
(\ref{Wend}) in \cite[Theorem 3.1]{ChHub}, for certain integer values of $\mu$ and $\alpha$.
They satisfy 
(\ref{fourier}) in \cite[Theorem 2.3]{ChHub} for fairly general choices of parameters $\mu, \alpha$.

 It follows from the definition that  ${\lambda_1}$ determines the order of the RBF in the sense that 
the reproducing kernel Hilbert space (the so-called native space)
it generates via the Moore-Aronszajn theorem
 is $W_2^t(\R^d)$.
It is shown in \cite[Section 5.2]{HR-Extending} that by considering longer series,
i.e., for greater polynomial degrees $\deg(p)$,
 in (\ref{Wend}), it is possible to increase the 
smoothness ${\lambda_2}$ near the support boundary boundary without changing 
the kernel's core 
analytic characteristics:
the global smoothness ${\lambda_1}$,
the Fourier decay parameter 
$t=({\lambda_1}+d+1)/2$,
or the native space $W_2^t(\R^d)$.

\section{Eccentric smoothing of compactly supported RBFs}

There are other ways to  increase $\lambda_2$ 
 without sacrificing 
the core analytic characteristics
-- we call this {\em eccentric smoothing} and consider it and its benefits below.

The construction below requires 
an RBF
$\phi$ satisfying 
Definition \ref{gen_wend}
and another, smoother, compactly supported RBF  $\tau$ which is otherwise arbitrary (it can be any
of the ones considered in \cite{Buhmann, Johnson, Schaback, Wend, Wu}
or any other method, provided the support and smoothness conditions are satisfied).

\begin{proposition}
\label{fattening}
Suppose $\phi$
is a positive definite RBF
satisfying Definition \ref{gen_wend}
with global smoothness $\lambda_1$ 
and eccentric smoothness
${\lambda_2}\ge 2t=\lambda_1+1+d$.
Suppose furthermore that $\tau$ 
is a positive definite RBF with 
support in $B(0,1)$: 
\begin{enumerate}
\item If 
$\tau \in C^k(\R^d)$ 
with 
$k\ge \max(2t,\lambda_2)$,
and $\int_{\R^d} \tau(x) \dif x >0$,
then 
$\tilde{\phi} 
= 
\tau  \phi$  
lies in 
$C^{{\lambda_1},1}(\R^d)$,
with 
$\tilde{ \phi}|_{\R_*^d}
\in 
C^{k}(\R_*^d)$;
\item If $\tau$ has the form given in 
Definition \ref{gen_wend}, 
with global smoothness 
$\kappa_1\ge\lambda_1$
and eccentric smoothness $\kappa_2$,
then 
$\tilde{\phi} = \tau  \phi$ 
lies in $C^{{\lambda_1},1}(\R^d)$,
with 
$\tilde{ \phi}|_{\R_*^d}
\in 
C^{\kappa_2+{\lambda_2}+1}(\R_*^d)
$.
 \end{enumerate}
In each case, 
the Fourier transform of 
$\tilde{\phi}$
satisfies
$$
\tilde{C}_1\llbracket \xi\rrbracket^{-2t}
\le 
\widehat{(\tau\phi)}(\xi)  
\le \tilde{C}_2\llbracket \xi\rrbracket^{-2t}
$$
for some $0 < C_1\le C_2$, 
and 
$\tilde{\phi}$ is positive definite.
\end{proposition}
Case 2 shows that the class of
functions satisfying Definition
\ref{gen_wend} is closed
under pointwise multiplication.
Consequently, computing an eccentrically smoothed RBF 
of this type is not harder
than computing the original 
function $\phi$. Indeed,
for a Wendland or generalized
Wendland function $\phi$, one
may take the pointwise 
power $\phi^n$; from the 
point of view of RBF interpolation
on a set $X$,
this is simply a matter of
taking $n$th powers of the matrix
entries of $\Phi_X$.
\begin{proof}
The fact that 
$
\tau \times \phi 
\in 
C^{{\lambda_1},1}(\R^d)
$ 
follows from the product rule.

The eccentric smoothness in case 1 
follows from the product rule, 
as well.

In case 2, $\tau$ and $\phi$ both 
vanish to high degree on 
$\partial B(0,1)$: 
$D^{\alpha} \bigl(\tau\phi\bigr)(x) 
= 
\sum_{\beta\le \alpha} 
C_{\beta,\alpha}
D^{\beta} \tau(x) 
D^{\alpha-\beta} \phi(x)$
by Leibniz's rule; 
if $|x|=1$, 
any such term sum is nonzero 
only if 
both $|\beta|>\kappa_2$ 
and $|\alpha-\beta|>\lambda_2$; 
thus
$D^{\alpha}\tilde{\phi}(x) =0$ 
for all $|x|=1$ and all 
$|\alpha|\le \kappa_2+\lambda_2+1$,
and the product is in 
$C^{\kappa_2+{\lambda_2}+1}$ 
in a neighborhood
of the sphere, 
and thus on $\R_*^d$.

The Fourier transform of 
$\tilde{\phi}$ is given by the 
convolution
$
\widehat{\tau \phi}(\xi) 
= [\hat{\tau}*\hat{\phi}](\xi)
$.
This  guarantees strict 
positivity of the Fourier transform 
and therefore  positive definiteness 
of $\tilde{\phi}$.
Furthermore,
Peetre's inequality guarantees the 
convolution estimate
$$
\int 
\llbracket\zeta\rrbracket^{-N_1}
\llbracket\xi-\zeta\rrbracket^{-N_2} \dif \zeta 
\le  
C 
\llbracket \xi
\rrbracket^{-\min(N_1,N_2)}$$
(see also
 \cite[Appendix B.1]{Grafakos_M}), 
 which we can use to get the desired 
 bounds on the Fourier transform of 
 $\tilde{\phi}$.

In each case, 
$\widehat{\tau}(\xi) 
\le 
C \llbracket \xi\rrbracket^{-2t}
$.
(In case 1, 
$\tau \in C_c^k(\R^d)$
guarantees  
$\widehat{\tau}(\xi) 
\le 
C \llbracket \xi\rrbracket^{-k}$
for some $C$,
while in case 2 
it follows from Definition \ref{gen_wend}
and 
the fact that 
$\kappa_1\ge \lambda_1$.)
An application of the 
above convolution inequality 
with $N_2= 2t = N_1$,
ensures the upper bound. 

For the lower bound, 
positivity and continuity of
$\hat{\tau}$ 
ensure the $\hat{\tau}(z)>c>0$ 
on a neighborhood
$B(0,\delta)$; 
thus we have 
$$ 
[\hat{\tau}*\hat{\phi}](\xi) 
= 
\int \hat{\phi}(\zeta) 
\hat{\tau}(\xi-\zeta)\dif \zeta
\ge 
\int_{|\xi-\zeta|<\delta}  
\hat{\phi}(\zeta) 
\hat{\tau}(\xi-\zeta)
\dif \zeta
\ge 
c 
\int_{|\xi-\zeta|<\delta} 
\hat{\phi}(\zeta)\dif \zeta
$$
Using the lower bound 
$\hat{\phi}(\zeta)
\ge 
C_1\llbracket \zeta \rrbracket^{-2t}
$
and noting there is a 
$\delta$-dependent constant $C'$ 
so that for 
$\zeta\in B(\xi,\delta)$, 
$\llbracket \zeta \rrbracket 
\le C' \llbracket \xi \rrbracket$, 
we have
$$ 
[\hat{\tau}*\hat{\phi}](\xi) 
\ge 
c 
\int_{|\xi-\zeta|<\delta} 
\hat{\phi}(\zeta)
\dif \zeta
\ge 
c C_1 (C')^{-2t} 
\mathrm{vol}(B(\xi,\delta)) 
\llbracket \xi \rrbracket^{-2t} .
$$
and the result follows with 
$\tilde{C}_1 =
c (C')^{-2t} 
\mathrm{vol}(B(\xi,\delta))>0.
$
\end{proof}
%
%
%
\section{Compressibility by samplets}
Let $\varrho$ be a positive integer.
A kernel $\kappa:\R^d\times \R^d\to \R$ is ``$\varrho$\,-\,asymptotically smooth'' in the sense of \cite[Section 3.3]{HMSS}
if for all multi-integers $\alpha, \beta$ for which $|\alpha|,|\beta|\le \varrho$, there exists a constant
$c_{\alpha,\beta}$ so that
\begin{equation}\label{eq:asympsmooth}
|D_x^{\alpha}D_y^{\beta} \kappa(x,y) | \le c_{\alpha,\beta} |x-y|^{-(|\alpha|+|\beta|)}.
\end{equation}
Kernels which are $\varrho$\,-\,asymptotically smooth have collocation matrices which can be nicely compressed
by using {\em samplets}  
of order $\varrho$, which we discuss below.

A samplet basis of order $\varrho$,
$$\Sigma  = \Phi_0 \cup \bigcup_{j=0}^J \Sigma_j\subset \mathrm{span}_{\xi\in\Xi} \delta_{\xi},$$ 
is a collection of finitely supported measures constructed by a multiresolution analysis (MRA)
 which is a basis for $\mathrm{span}\{ \delta_{\xi}\mid \xi\in\Xi\}$, and for which
 every samplet $\sigma_{j,k} \in \Sigma_j$  annihilates $\mathcal{P}_{\varrho-1}(\R^N)$
  (i.e., has $\varrho-1$ vanishing moments).
 Indeed, $\Sigma$ is an orthonormal basis in the induced $\ell_2(\Xi)$ inner product given
 by $\|\sum_{\xi\in \Xi} a_{\xi} \delta_{\xi}\|^2 = \sum_{\xi\in\Xi} |a_{\xi}|^2$.
 
 Associated to each samplet $\sigma_{j,k} $ is a {\em cluster} $\nu_{j,k} \subset \Xi$ 
 with $\supp(\sigma_{j,k})\subset \nu_{j,k}$,
 The clusters (and therefore the supports) have the property that for every $j,k$ and $j',k'$ with $j'\ge j$,
 $\nu_{j',k'}\subset \nu_{j,k}$ or $\mathrm{conv}(\nu_{j,k}) \cap \mathrm{conv}(\nu_{j',k'}) =\emptyset$. 
 Indeed, the clusters are the nodes of a binary tree structure with $\nu_0 = \Xi$ at the root, and were every 
 where every non-leaf node $\nu_{j,k}$ has precisely two children $\nu_{j+1,k_1}$ and $\nu_{j+1,k_2}$
 which satisfy $\nu_{j,k} =\nu_{j+1,k_1} \cup \nu_{j+1,k_2}$.
 See \cite{HM} for the construction  of the cluster three, the  MRA, as well as presentation of desirable properties
 of the samplets.
 
 The standard collocation matrix $K_{\Xi} =\bigl(\kappa(\xi,\zeta)\bigr)_{\xi,\zeta\in\Xi} = (\langle \kappa, \delta_{\xi}\otimes \delta_{\zeta}\rangle)_{\xi,\zeta\in\Xi}$ can be easily transformed via a change of basis to 
 $
 K_{\Xi}^{\Sigma}
 =
\langle 
 \kappa,\sigma\otimes\sigma'
\rangle_{\sigma,\sigma' \in \Sigma}
$ (this change of basis is unitary, by the orthonormality of the samplets).
For a  $\varrho$-asymptotically smooth 
kernel, the collocation matrix  in 
samplet coordinates
$K_{\Xi}^{\Sigma}$ 
can be nicely compressed by
\cite[Theorem 3]{HMSS}, 
 to produce a sparse matrix.
 For a user defined $\eta>0$, 
 the  matrix $K_{\eta}$ 
 is obtained by 
 truncating
$K_{\Xi}^{\Sigma}$ 
 by setting
 $
 \bigl({K}_{\eta}\bigr)_{\sigma,\sigma'}
 =
 0
 $
 if  the clusters $\nu$ and $ \nu'$ of $\sigma$ and $\sigma'$ (respectively)
 are distant in the sense that 
 $\mathrm{dist}(\nu,\nu')
 \ge 
 \eta \max(\mathrm{diam}(\nu),\mathrm{diam}(\nu'))$.  
 For this setup, there is a constant $c$ so that
$$\|K_{\Xi}^{\Sigma}-K_{\eta}\|_{F} \le c_{\text{samp}}\eta^{-2\varrho} N \log(N), \qquad \text{where }N =\#\Xi.$$
This is precisely \cite[Theorem 3]{HMSS};
furthermore, \cite[Corollary 1]{HMSS} shows that for quasi-uniform sets $\Xi$,
a relative error of $\frac{\|K_{\Xi}^{\Sigma}-K_{\eta}\|_{F}} {\|K_{\Xi}^{\Sigma}\|_{F}}\le c\eta^{-2\varrho}$
with compressed matrix $K_{\eta}$ having $\mathcal{O}(N \log N)$ nonzero entries. 

Naturally,
for such estimates to be useful, it makes sense to assume $\eta>1$,
since the case $\eta\le 1$ in the estimates above can already be accomplished 
by setting $K_{\eta}=0$.
(Note, for instance, that the uniform bound on the diagonal entries $\sup_{\xi\in \R^d} \kappa(\xi,\xi) \le C$
guarantees that  the entries of $K_{\Xi}^{\Sigma}$ are bounded, and so $\|K_{\Xi}^{\Sigma}\|_{F}
= \|K_{\Xi}\|_F \le C N$ ).

For translation invariant kernels like  compactly supported RBFs, where $\kappa(x,y) = \phi(x-y)$,
the above condition reduces to the following: for all $|\alpha|\le 2\varrho$, there is $c_{\alpha}$
so that for all $x\neq 0$,
$$|D_x^{\alpha} \phi(x) | \le c_{\alpha} |x|^{-|\alpha|}.$$
Naturally this means 
$\phi |_{\R_*^d} 
\in C^{2\varrho}(\R_*^d)$ -- 
in short, the asymptotic smoothness
of  a compactly supported RBF is limited by its smoothness at the frontier $\partial B(0,1)$.
The originally considered RBFs 
satisfying Definition \ref{gen_wend}
are at most asymptotically smooth of 
order $\lambda_2/2$.

On
the other hand, as we show below, eccentric smoothing a compactly supported kernel provides greater
asymptotic smoothness.
It follows that compressibility improves for eccentrically smoothed  RBFs. 

Since $\lambda_2>\lambda_1+1$, one could
simply take a higher order compactly supported RBF 
(one with greater global smoothness,  a more rapidly decaying Fourier transform and a smaller native space), 
but this comes at a cost to the condition number  of the matrix $\Phi_{\Xi}$ (as well as its representation
in the samplet basis $\Phi_{\Xi}^{\Sigma}$).
 In particular, \cite[Theorem 12.3]{Wendland_book}  ensures that the smallest eigenvalue of $\Phi_{\Xi}$
behaves like 
$\lambda_{\min} (\Phi_{\Xi}) 
\ge c_{\text{eig}} q_{\Xi}^{2t-d} = c_{\text{eig}} q_{\Xi}^{\lambda_1+1}$, where $t = (\lambda_1+1+d)/2$ and 
$q_{\Xi}=\min_{\xi\in\Xi}\min_{ \zeta\in \Xi\setminus\{\xi\}} |\xi-\zeta|$.
This leads to a theoretically bounded condition number 
on the order $ q_{\Xi}^{-2t}$ for point sets exhibiting the bound $N=\#\Xi \sim q_{\Xi}^{-d}$ 
(this is a consequence of {\em quasi-uniformity}).

Thus one virtue 
of eccentric smoothing of the compactly supported kernel 
is that it improves the compressibility of the
collocation matrix  while preserving  its conditioning.

\paragraph{Asymptotic smoothness of finite order}
Any compactly supported
kernel of the form (\ref{Wend})
with $\phi |_{\R_*^N} \in C^{2\varrho}(\R_*^N)$ is $\varrho$-asymptotically smooth. 
In particular, we have the following.
\begin{cor}
If $\tilde{\phi}$ is the eccentrically smoothed  RBF considered in case 1 of Proposition \ref{fattening}
then $\tilde{\phi}$ is 
$k/2$
asymptotically smooth.
\end{cor}
\begin{proof}
The fact that 
 $\tilde{\phi}|_{\R_*^d}
 \in C^{k}$,
 ensures that there is $C_{\phi}$ so 
 that  for any 
 $|\alpha|\le k$ and  
 for any $|x|\ge 1/2$, 
 $$|D^{\alpha}\tilde{\phi}(x)|
 \le  C_{\phi} |x|^{-|\alpha|}$$
 (this uses the fact that $\tilde{\phi}$ is supported in the unit ball).
 
To obtain bounds near to the origin, note that homogeneity implies  that,
for any $b\in \R$ and $\alpha\in \mathbb{Z}_+^d$, the
estimate 
$\bigl|D^{\alpha} |x|^{b}\bigr| 
\le 
\gamma_{b,\alpha} |x|^{b-|\alpha|}
$ holds
for some constant $\gamma_{b,\alpha} $ and for all $x\neq 0$. 
Consequently, for any multi-index $\alpha$ there is $C_{\alpha}$ so that if $|x|\le 1$ then
 $$
 |D^{\alpha}\phi(x)|
 \le 
 \sum_{\ell=0}^M 
 |c_{\ell}| 
 \gamma_{\lambda_1+1+2\ell,\alpha} 
 |x|^{\lambda_1+1+2\ell-|\alpha|} 
 +\bigl|D^{\alpha} q(|x|^2)\bigr| 
 \le C_{\alpha} |x|^{-|\alpha|}.
 $$
 Multiplying $\phi$ by $\tau$, which is $C^k$ gives, by Leibniz's rule, that
 $$|D^{\alpha}\tilde{\phi}(x)|\le \sum_{\beta\le \alpha} \begin{pmatrix} \alpha\\ \beta\end{pmatrix} \|D^{\beta} \tau\|_{\infty} C_{\alpha-\beta} |x|^{-|\alpha-\beta|}
 \le C|x|^{-|\alpha|}$$
 provided $|\alpha|\le k$.
\end{proof}

We can do better than this if we choose $\tau$ to also be a generalized  Wendland function, as given
in Definition \ref{gen_wend}; in particular, by requiring $\tau$ to be polyhomogeneous on its support. 
 
 \begin{cor}
 \label{Wend_smoothing}
 If $\tilde{\phi} = \tau\phi$ 
 is a eccentrically smoothed  RBF
considered 
in case 2 of
Proposition \ref{fattening},
 then  $\tilde{\phi}$ is $(\kappa+\lambda_2+1)/2$ asymptotically smooth.
\end{cor}

\begin{proof} 
Since the product of polyhomogeneous 
functions is polyhomogeneous,
we have for $|x|\le 1$,
$$\tilde{\phi}(x) = 
\sum_{\ell =0}^{\tilde{M}} 
a_{\ell} |x|^{2a+1+2\ell} 
+ \tilde{q}(|x|^2).
$$
Consequently, 
for all $0<|x|<1$, 
$|D^{\alpha} \tilde{\phi} (x) |
\le C |x|^{-|\alpha|}$ 
for some constant $C$
depending on $\alpha$, 
and  the coefficients $a_{\ell}$ and $\tilde{q}$ 
(in particular, 
this holds for all $|\alpha|\ge 0$).
On the other hand, 
$\tilde{\phi} \in C^{\kappa_2+\lambda_2+1}(\R_*^d)$ 
and vanishes outside of $B(0,1)$, 
so the estimate 
$
|D^{\alpha} \tilde{\phi} (x) |
\le C |x|^{-|\alpha|}
$ 
extends to $\R_*^d$ as long as 
$|\alpha|\le \kappa_2+\lambda_2+1$.
\end{proof}

One application of this is the following. Suppose, one is interested in solving a linear system with the compressed matrix $K_{\Xi}^{\Sigma} x\approx K_{\eta} \tilde{x} =f$ in order to obtain an approximate solution. In order to solve the system with the compressed matrix, one needs that $K_{\eta}$ is regular (and by construction symmetric), i.e., that its lowest eigenvalue is bounded below. 
This follows by classical perturbation of eigenvalue results from
\begin{equation*}
    \|K_{\Xi}^{\Sigma}-K_{\eta}\|_{2} \le 
    c_{\text{samp}}\eta^{-2\varrho} q_{\Xi}^{-d} d | \log(q_{\Xi})| 
    \le 
    c_{\text{samp}}\eta^{-2\varrho} q_{\Xi}^{-d} d q^{-1}_{\Xi}
    \le 
    \frac{1}{2}c_{\text{eig}} q_{\Xi}^{\lambda_1+1} 
    {\le}
    \frac{1}{2}\lambda_{\min} (\Phi_{\Xi}). 
\end{equation*}
Now, using Corollary 
\ref{Wend_smoothing} with
$2\varrho=\kappa+\lambda_2+1$, 
we obtain that $K_{\eta}$ is positive definite for 
$ \eta\gtrsim 
q_{\Xi}^{-\frac{\lambda_1+d+2}{\kappa+\lambda_2+1}}$.
As $\eta$ determines the sparsity of $K_{\eta}$, the inequality ensuring positive definiteness might also be used to determine the eccentric smoothing parameter $\kappa$ if $\eta$ is considered fixed by enforcing enough compression to make matrix operations feasible.

\section{Compactly supported RBFs with infinite eccentric smoothness}

It follows that there exists for any $\lambda_2>2t$ 
a compactly supported RBF $\phi$
with native space $W_2^t(\R^d)$ and boundary smoothness 
$\phi|_{\R_*^d}\in C^{\lambda_2}(\R_*^d)$.
Notably, by considering the standard mollifier 
$\varphi\in C_c^{\infty}(\R^d)$ 
with support in $B(0,1/2)$
the iterated function 
$$
\tau 
= 
\varphi*\varphi
$$
 has Fourier transform
$\widehat{\tau}(\xi) = 
|\widehat{\varphi}(\xi)|^2
\ge 0
$ which is analytic and thus has
only isolated zeros on $\R^d$.
Consequently, $\tau$ is an infinitely smooth RBF with support in $B(0,1)$.
By using this function in Proposition \ref{fattening}, we can obtain compactly supported
RBFs with  Sobolev native space $W_2^t(\R^d)$ and singleton singular support:
$\mathrm{sing \,supp}(\phi )= \{0\}$.
See for instance \cite{platte:2015} for 
an exposition on how to compute these 
kernels numerically as no closed-form 
expression seems to be available. See 
also the recent paper 
\cite{Schaback:2026} 
for a Mercer type approach towards 
smooth, compactly supported kernels, 
which are not radial.
\begin{proposition} 
\label{symbol}
Let  $\phi$ be a compactly supported RBF 
of the form (\ref{Wend}), and let
$\tilde{\phi}= \phi*\tau$ be a compactly 
supported RBF
obtained from Proposition \ref{fattening} 
with $\tau \in C^{\infty}(\R^d)$.
Then the integral operator 
$\tilde{\phi}\!*$ is pseudodifferential 
with 
symbol in $S_{1,0}^{-2t}(\R^d)$. 
\end{proposition}
\begin{proof}
Fix $\alpha$
It suffices to show that 
$|D^{\alpha}\mathcal{F}\tilde{\phi}(\xi) |\le C\llbracket \xi\rrbracket^{-(2t+|\alpha|)}$
for all $\xi$. Since $\mathcal{F}\tilde{\phi}(\xi)$ is continuous, it suffices to demonstrate
 this inequality for $|\xi|$ sufficiently large (we will describe this in the next paragraph).

This follows because  
$\Phi:\R^d\to \R$, defined by
$\Phi(x) =\sum_{\ell=0}^M c_\ell |x|^{\lambda_1+1+2\ell} +q(|x|^2)$
for all $x$ (not just for $|x|\le 1$),
as a combination of homogeneous functions,
satisfies the following: there exist constants  $\Gamma_{\alpha}$ and $R_{\alpha}>0$  
so that  for any  $|\xi|\ge R_\alpha$,
the inequality
\begin{equation}
\label{polyharmonic}
|D_{\xi}^{\alpha} \widehat{\Phi}(\xi)|
\le \Gamma_{\alpha} \llbracket \xi\rrbracket^{-2t-|\alpha|}
\end{equation}
holds.
Set $|\xi|>4R_\alpha$.

Since $\tau$ is supported on $B(0,1)$,
the  functions $ \phi\tau $ and
$  \Phi \tau$ are identical.
Thus $\tilde{\phi}=\Phi \tau$ has distributional Fourier transform
$\mathcal{F} \tilde{\phi}= [\widehat{\Phi}* \widehat{\tau}]$ which satisfies for $\xi\neq 0$
$$
D^{\alpha}\mathcal{F}\tilde{\phi}(\xi)  = [\widehat{\tau}* D^{\alpha} \widehat{\Phi}](\xi)
=
\langle D^{\alpha} \widehat{\Phi}(\xi - \cdot),\widehat{ \tau}\rangle 
$$
Using a smooth cut-off $\upsilon$, 
supported in $B(0,2R_{\alpha})$ and equaling
one on $B(0,R_{\alpha})$,
we express the above pairing as
$$
D^{\alpha}\mathcal{F}\tilde{\phi}(\xi) = 
\langle D^{\alpha} \widehat{\Phi}(\xi - \cdot),G\rangle 
+
\langle D^{\alpha} \widehat{\Phi}(\xi - \cdot),B\rangle$$
where 
$B=\widehat{ \tau}\upsilon(\xi-\cdot)$ and 
$G = \widehat{ \tau}\bigl(1-\upsilon(\xi-\cdot)\bigr)$
are Schwartz functions, 
having supports
$\mathrm{supp}(B)\subset B(\xi,2R_\alpha)$ and $\mathrm{supp}(G)\subset \R^d\setminus B(\xi,R_\alpha)$.

To treat $|\langle D^{\alpha} \widehat{\Phi}(\xi - \cdot),B\rangle|$, we note that 
because $ D^{\alpha} \widehat{\Phi}(\xi - \cdot)$ is a tempered distribution,
$$|\langle D^{\alpha} \widehat{\Phi}(\xi - \cdot),B\rangle|
\le 
C_{\Phi}\sum_{|\alpha|\le N,|\beta|\le M} \rho_{\alpha,\beta}(B)$$ 
for some constant $C_{\Phi}$
and for semi-norms $\rho_{\gamma,\beta} (B)= \sup_{\xi \in \R^d}|x^{\gamma}D^{\beta}B(\xi)|$.
Since $\widehat{\tau}\in \mathcal{S}(\R^d)$,
the fact that 
  $$\sup_{\zeta \in B(\xi,2R_{\alpha})}  |\zeta^{\gamma} D^{\beta}\widehat{\tau}(\zeta)| \le C_{\gamma,\beta} \min_{\zeta\in B(\xi,2R_{\alpha})}|\zeta|^{-M}$$ 
  for each $\gamma,\beta$ implies that for any $M>0$ there is $C_M$ so that 
$$\sum_{|\alpha|\le N,|\beta|\le M} \rho_{\alpha,\beta}(B) 
\le C_M (1 +|\xi| -2R_{\alpha})^{-M}.$$
Thus
$$|\langle D^{\alpha} \widehat{\Phi}(\xi - \cdot),B\rangle|
\le C_{\Phi} \sum_{|\alpha|\le N,|\beta|\le M} \rho_{\alpha,\beta}(B)
\le  
C_{\Phi} C_{M} (1+|\xi|-2R_\alpha)^{-M}\le C \llbracket \xi\rrbracket^{-M} .$$
for any $M$ -- in particular, for $M>2t+|\alpha|$.

To treat $\langle D^{\alpha} \widehat{\Phi}(\xi - \cdot),G\rangle $, we note that the pairing is represented by an integral, since $\widehat{\Phi}$ is regular
 on $\R_*^d$. Thus, for $M>2t+|\alpha|+d$, we have
\begin{eqnarray*}
|\langle D^{\alpha} \widehat{\Phi}(\xi - \cdot),G\rangle |
&=&
\left|\int_{|\xi-\zeta|>R_\alpha} 
  D^{\alpha} \widehat{\Phi}(\xi - \zeta) 
  \widehat{\tau}(\zeta) 
  \bigl(1-\upsilon(\xi-\zeta)\bigr)
\dif\zeta\right|\\
&\le&
\Gamma_{\alpha}C_M
\int_{|\xi-\zeta|>R_{\alpha}} 
 \llbracket \xi -\zeta\rrbracket^{-2t-|\alpha|}
  \llbracket\zeta\rrbracket^{-M} 
\dif\zeta.
\end{eqnarray*}
 Here we have used (\ref{polyharmonic}) and the fact that $\tau$ is a Schwartz function. 
 To finish the estimate,
 we use Peetre's inequality 
 $ \llbracket\xi -\zeta\rrbracket^{-2t-|\alpha|} \le 2^{2t+|\alpha|}  \llbracket \xi \rrbracket^{-2t-|\alpha|} 
 \llbracket \zeta \rrbracket^{2t+|\alpha|}$ 
 to get
$$ |\langle D^{\alpha} \widehat{\Phi}(\xi - \cdot),G\rangle |
\le C \llbracket \xi \rrbracket^{-2t-|\alpha|} .$$
Adding this to the bound for 
$|\langle D^{\alpha} \widehat{\Phi}(\xi - \cdot),B\rangle|$ gives 
$$|D^{\alpha}\mathcal{F}\tilde{\phi}(\xi) |=
\bigl|D^{\alpha}[\widehat{\Phi}* \widehat{\tau}](\xi)\bigr|\le C \llbracket \xi \rrbracket^{-2t-|\alpha|}$$
and the result follows.
\end{proof}

\paragraph{Mapping properties}
We now consider mapping properties of the induced integral operator -- for this we 
consider the scale of Besov spaces $B_{p,q}^s(\R^d)$.
What we present here can easily be modified
to handle Triebel-Lizorkin spaces $F_{p,q}^s(\R^d)$; 
 standard Sobolev spaces $W_p^s(\R^d)$ for $1<p<\infty$ are either Besov (for fractional spaces)
 or Triebel-Lizorkin (for integer spaces); the cases $p=1,\infty$  in the Triebel-Lizorkin scale are more difficult.

To define Besov spaces, we begin by fixing a spectral partition of unity.
Let $\psi:\R^d\to [0,\infty)$ be a bandlimited (Schwartz) test function with annular support
$\mathrm{supp}(\widehat{\psi})\subset B(0,2)\setminus B(0,1/2)$ and which satisfies
$\sum_{k=-\infty}^{\infty} \widehat{\psi}(2^k \xi) =1$ for all $\xi\in\R^{d}\setminus\{0\}$.
Define 
\begin{equation}
\label{fourier_multiplier}
\psi_k:=2^{dk} \psi(2^{k} \cdot)
\end{equation}
 to be a dyadic dilation of $\psi$. 
 Define $\omega \in \mathcal{S}(\R^d)$ by
 $\widehat{\omega}= 1-\sum_{k=1}^{\infty}\widehat{ \psi}(2^{-k}\cdot)$

\begin{definition}
For $s\in \R$ $p,q\in[1,\infty]$, the Besov space $B_{p,q}^{s}(\R^d)$ 
consists of tempered distributions, $f$, for which $f*\omega\in L_p$, each $\psi_k*f\in L_p$
and for which the norm
\begin{equation*}
\| f\|_{B_{p,q}^{s}(\R^d)}:= \|f*\omega\|_{L_p(\R^d)}+ \Bigl\|  k\mapsto 2^{ks} \|\psi_k*f\|_{L_p(\R^d)} \Bigr\|_{\ell_q(\N)}
\end{equation*}
is finite.  
\end{definition}

\begin{cor}
\label{besov}
Let $\tilde{\phi}$ be
the eccentrically smoothed, compactly supported RBF
of Proposition \ref{symbol}.
Then for any $p,q\in [1,\infty]$ 
and any $s\in \R$, the map
$$\tilde{\phi}*:
B_{p,q}^s(\R^d)\to B_{p,q}^{s+2t}(\R^d):
f\mapsto \tilde{\phi}*f
$$
is bounded.
\end{cor}
\begin{proof}
Since $\tilde{\phi}*$ is proper, 
its composition  with other 
pseudodifferential operators 
is well-defined. 
In particular, for the
Bessel potential operator,
$J^{2t}:B_{p,q}^s(\R^d)
\mapsto B_{p,q}^{s-2t}(\R^d)$ 
defined by 
$(J^s f)^{\wedge} (\xi)
= 
(1+|\xi|^2)^{s/2}\widehat{f}(\xi)$
(see \cite[Lemma 6.2.1]{BL}
or \cite[Remark 3, Section 6.2.1]{Trieb})
we have that 
$J^{2t}\circ \tilde{\phi}*$ is a 
pseudodifferential operator in 
$S_{1,0}^{0}(\R^d).$ 

The result
follows as a direct consequence
of regularity properties of 
H{\"o}rmander class pseudodifferential
operators; see for instance
\cite[Theorem 6.2.2]{Trieb}:
since the composition
$J^{2t}\circ \tilde{\phi}*$ 
is a pseudodifferential operator
in $S_{1,0}^0$, the map
$J^{2t}\circ \tilde{\phi}*:B_{p,q}^{s}(\R^d)\to B_{p,q}^s(\R^d)$ 
is continuous,
the composition
$\phi*  =J^{-2t} \circ J^{2t}\circ \tilde{\phi}*: B_{p,q}^{s}(\R^d)\to B_{p,q}^{s+2t}(\R^d)$
is continuous as well.
\end{proof}

Let $\mu$ be a finite, signed Borel measure on $\R^d$, with Jordan decomposition $\mu= \mu^+-\mu^-$,
 total variation measure $|\mu| =  \mu^+ + \mu^-$ and total variation norm $\|\mu\|_{TV} = |\mu|(\R^d)$.
Then  Young's theorem gives 
$\|\mu*f\|_p\le \|\mu\|_{TV} \|f\|_p$.
This implies $\|\mu *\psi_k\|_p\le 2^{kd/p'}\|\psi\|_p \|\mu\|_{TV}$,
where $p'$ is the conjugate exponent to $p$, 
i.e., $1/p+1/p'=1$. Thus $\mu\in B_{p,\infty}^{-d/p'}(\R^d)$ and 
\begin{equation}
\label{finite_measure_besov}
 \|\mu\|_{B_{p,\infty}^{-d/p'} (\R^d)}\le C_{\psi} \|\mu\|_{TV} 
\end{equation}
holds with a constant
$C_{\psi}$ depending on the
spectral
partition of unity 
induced by $\psi$.

\begin{cor}
Let $\tilde{\phi}$ be
the eccentrically smoothed, compactly supported RBF
of Proposition \ref{symbol}.
If $\mu$ is a finite, signed Borel measure on $\R^d$
then the integral operator
$$\mathcal{L}_{\tilde{\phi}}:L_p(\mu) \to S'(\R^d): f\mapsto \int_{\R^d} f(y) \tilde{\phi}(\cdot-y)\dif \mu(y)$$
has range in $B_{p,\infty}^{2t-d/p'} (\R^d)$ and
$\mathcal{L}_{\tilde{\phi}}: L_p(\mu) \to B_{p,\infty}^{2t-d/p'} (\R^d)$ is continuous.
\end{cor}

\begin{proof}
Since
$\mu\in B_{p,\infty}^{-d/p'}(\R^d)$, 
where $p' = (1-1/p)^{-1}$,
it follows that $L_p(\mu)\subset B_{p,\infty}^{-d/p'}(\R^d)$ as well,
in the sense that 
$
f\times \mu:
U\mapsto \int_{U} f(x) \dif \mu(x)
$ 
is a finite measure  
and
$
\|f\times \mu\|_{TV} 
\le 
\|f\|_{L_p(\mu)} \|\mu\|_{TV}^{1/p'}
$. 
Indeed, for any $g\in C_0(\R^d)$,
$$
\left|
\int_{\R^d} g(x) f(x) 
\dif \mu(x)
\right|
\le 
\|g\|_{\sup} 
\int_{\R^d}
|f(x)|\, 
\dif |\mu|(x)
\le 
\|g\|_{\sup} \|f\|_{L_p(\mu)} \bigl(
|\mu|(\R^d)
\bigr)^{1/p'}.
$$
Consequently,
$\|f\times \mu\|_{B_{p,\infty}^{-d/p'}(\R^d)} 
\le 
C_{\psi}\|f\times \mu\|_{TV} 
\le 
C_{\psi}
\|f\|_{L_p(\mu)} \|\mu\|_{TV}^{1/p'}$
holds.

It follows that $f\times \mu$ 
is a tempered distribution and 
 that 
$
\tilde{\phi}* (f\times \mu) 
 = 
\int_{\R^d} f(y) \tilde{\phi}(x-y)
 \dif \mu(y)
$.
 By Corollary \ref{besov},
$$
\left\|
\int_{\R^d} f(y) 
\tilde{\phi}(\cdot-y) 
\dif \mu(y)
\right\|_{B_{p,\infty}^{2t-d/p'} (\R^d)}
\le 
C_{\psi}
\|f\times \mu\|_{B_{p,\infty}^{-d/p'}(\R^d)}
\le {C} \|f\|_{L_p(\mu)}  \|\mu\|_{TV}^{1/p'},$$
where
$C$ incorporates
the constants from (\ref{finite_measure_besov})
and the operator
norm 
$\|\tilde{\phi}*\|_{B_{p,\infty}^{-d/p'}(\R^d) 
\to B_{p,\infty}^{2t-d/p'}(\R^d)}$.
\end{proof}

The above situation may be improved for specific $\mu$ (for
instance if $\mu$ is Hausdorff measure of some compact set); in that case the range may be in a 
smaller Besov space with smoothness index greater than $2t-d/p'$ -- \cite{HRW}.
Secondly, if $\mu$ is a Borel measure with infinite variation, then the above holds for $p=1$ (since
in that case, $f\times \mu$ is a finite measure -- this is the case for Lebesgue measure).


\bibliographystyle{plain}
\bibliography{literature}

\begin{thebibliography}{10}

\bibitem{Johnson}
Amal Al-Rashdan and Michael~J. Johnson.
\newblock Minimal degree univariate piecewise polynomials with prescribed
  sobolev regularity.
\newblock {\em Journal of Approximation Theory}, 164(1):1--5, 2012.

\bibitem{AKMW}
Sara Avesani, R\"udiger Kempf, Michael Multerer, and Holger Wendland.
\newblock Multiscale {S}cattered {D}ata {A}nalysis in {S}amplet {C}oordinates.
\newblock {\em SIAM J. Sci. Comput.}, 47(5):A3038--A3063, 2025.

\bibitem{BL}
J\"oran Bergh and J\"orgen L\"ofstr\"om.
\newblock {\em Interpolation spaces. {A}n introduction}, volume No. 223 of {\em
  Grundlehren der Mathematischen Wissenschaften}.
\newblock Springer-Verlag, Berlin-New York, 1976.

\bibitem{Berlinet}
Alain Berlinet and Christine Thomas-Agnan.
\newblock {\em Reproducing kernel {H}ilbert spaces in probability and
  statistics}.
\newblock Kluwer Academic Publishers, Boston, MA, 2004.
\newblock With a preface by Persi Diaconis.

\bibitem{Buhmann}
M.~D. Buhmann.
\newblock Radial functions on compact support.
\newblock {\em Proc. Edinburgh Math. Soc. (2)}, 41(1):33--46, 1998.

\bibitem{ChHub}
Andrew Chernih and Simon Hubbert.
\newblock Closed form representations and properties of the generalised
  {W}endland functions.
\newblock {\em J. Approx. Theory}, 177:17--33, 2014.

\bibitem{CS}
Felipe Cucker and Steve Smale.
\newblock On the mathematical foundations of learning.
\newblock {\em Bull. Amer. Math. Soc. (N.S.)}, 39(1):1--49, 2002.

\bibitem{Grafakos_M}
Loukas Grafakos.
\newblock {\em Modern {F}ourier analysis}, volume 250 of {\em Graduate Texts in
  Mathematics}.
\newblock Springer, New York, third edition, 2014.

\bibitem{Hackbusch}
Wolfgang Hackbusch.
\newblock {\em Hierarchical matrices: algorithms and analysis}, volume~49 of
  {\em Springer Series in Computational Mathematics}.
\newblock Springer, Heidelberg, 2015.

\bibitem{HR-Extending}
T.~Hangelbroek and C.~Rieger.
\newblock Extending error bounds for radial basis function interpolation to
  measuring the error in higher order {S}obolev norms.
\newblock {\em Math. Comp.}, 94(351):381--407, 2025.

\bibitem{HRW}
T.~Hangelbroek, C.~Rieger, and G.~B. Wright.
\newblock Kernel approximation beyond the native space -- with applications to
  approximation on manifolds, 2026.
\newblock preprint arXiv:2606.28564.

\bibitem{HMSS}
H.~Harbrecht, M.~Multerer, O.~Schenk, and Ch. Schwab.
\newblock Multiresolution kernel matrix algebra.
\newblock {\em Numer. Math.}, 156(3):1085--1114, 2024.

\bibitem{HM}
Helmut Harbrecht and Michael Multerer.
\newblock Samplets: construction and scattered data compression.
\newblock {\em J. Comput. Phys.}, 471:Paper No. 111616, 23, 2022.

\bibitem{HubJag}
Simon Hubbert and Janin J\"ager.
\newblock Closed form representations for the compactly supported radial basis
  functions of {B}uhmann, {W}endland and {W}u.
\newblock {\em Adv. Comput. Math.}, 51(6):Paper No. 48, 30, 2025.

\bibitem{LGZ}
Shao-Bo Lin, Xin Guo, and Ding-Xuan Zhou.
\newblock Distributed learning with regularized least squares.
\newblock {\em J. Mach. Learn. Res.}, 18:Paper No. 92, 31, 2017.

\bibitem{platte:2015}
Rodrigo~B. Platte.
\newblock {$C^{\infty}$} compactly supported and positive definite radial
  kernels.
\newblock {\em SIAM Journal on Scientific Computing}, 37(4):A1934--A1956, 2015.

\bibitem{schaback_doubling}
R.~Schaback.
\newblock Improved error bounds for scattered data interpolation by radial
  basis functions.
\newblock {\em Math. Comp.}, 68(225):201--216, 1999.

\bibitem{Schaback}
Robert Schaback.
\newblock The missing {W}endland functions.
\newblock {\em Adv. Comput. Math.}, 34(1):67--81, 2011.

\bibitem{Schaback:2026}
Robert Schaback.
\newblock Kernel construction techniques.
\newblock {\em Engineering Analysis with Boundary Elements}, 188:106782, 2026.

\bibitem{Taylor2}
Michael~E. Taylor.
\newblock {\em Partial differential equations {II}. {Q}ualitative studies of
  linear equations}, volume 116 of {\em Applied Mathematical Sciences}.
\newblock Springer, Cham, third edition, 2023.

\bibitem{Trieb}
Hans Triebel.
\newblock {\em Theory of function spaces. {II}}, volume~84 of {\em Monographs
  in Mathematics}.
\newblock Birkh\"auser Verlag, Basel, 1992.

\bibitem{Wend}
Holger Wendland.
\newblock Piecewise polynomial, positive definite and compactly supported
  radial functions of minimal degree.
\newblock {\em Adv. Comput. Math.}, 4(4):389--396, 1995.

\bibitem{Wendland_book}
Holger Wendland.
\newblock {\em Scattered data approximation}, volume~17 of {\em Cambridge
  Monographs on Applied and Computational Mathematics}.
\newblock Cambridge University Press, Cambridge, 2005.

\bibitem{Wu}
Zong~Min Wu.
\newblock Compactly supported positive definite radial functions.
\newblock {\em Adv. Comput. Math.}, 4(3):283--292, 1995.

\end{thebibliography}

\end{document}